\documentclass[11pt]{article}

\usepackage{amsfonts,amsmath,amsthm,latexsym,color,epsfig,hyperref,enumitem,calc}
\newtheorem{theorem}{Theorem}[section]
\newtheorem*{theorem*}{Theorem}
\newtheorem{lemma}{Lemma}[section]
\newtheorem{proposition}{Proposition}[section]
\newtheorem{conjecture}{Conjecture}[section]
\newtheorem{corollary}{Corollary}[section]

\def\qed{\ifvmode\mbox{ }\else\unskip\fi\hskip 1em plus 10fill$\Box$}

\input{epsf}
\makeatletter
\def\Ddots{\mathinner{\mkern1mu\raise\p@
\vbox{\kern7\p@\hbox{.}}\mkern2mu
\raise4\p@\hbox{.}\mkern2mu\raise7\p@\hbox{.}\mkern1mu}}
\makeatother

\title{\vspace{-0.7cm}Hereditary quasirandomness without regularity}
\author{David Conlon\thanks{Mathematical Institute, Oxford OX2 6GG,
United Kingdom. Email: {\tt david.conlon@maths.ox.ac.uk}. Research
supported by a Royal Society University Research Fellowship and by ERC Starting Grant 676632.}\and
Jacob Fox\thanks{Department of Mathematics, Stanford University, Stanford, CA 94305. Email: {\tt jacobfox@stanford.edu}. Research supported by a Packard Fellowship, by NSF Career Award DMS-1352121 and by an Alfred P. Sloan Fellowship.}
\and
Benny Sudakov\thanks{Department of Mathematics, ETH, 8092 Zurich, Switzerland.
Email: {\tt benjamin.sudakov@math.ethz.ch}. Research supported by SNSF grant 200021-149111.}}
\date{}

\begin{document}
\maketitle

\begin{abstract}
A result of Simonovits and S\'os states that for any fixed graph $H$ and any $\epsilon > 0$ there exists $\delta > 0$ such that if $G$ is an $n$-vertex graph with the property that every $S \subseteq V(G)$ contains $p^{e(H)} |S|^{v(H)} \pm \delta n^{v(H)}$ labeled copies of $H$, then $G$ is quasirandom in the sense that every $S \subseteq V(G)$ contains $\frac{1}{2} p |S|^2 \pm \epsilon n^2$ edges. The original proof of this result makes heavy use of the regularity lemma, resulting in a bound on $\delta^{-1}$ which is a tower of twos of height polynomial in $\epsilon^{-1}$. We give an alternative proof of this theorem which avoids the regularity lemma and shows that $\delta$ may be taken to be linear in $\epsilon$ when $H$ is a clique and polynomial in $\epsilon$ for general $H$. This answers a problem raised by Simonovits and S\'os.
\end{abstract}

\section{Introduction}

What does it mean to say that a graph is random-like and how does one construct such graphs? Attempts to answer these questions have played a central role in mathematics over the last forty years, with connections to combinatorics, probability, number theory, theoretical computer science and more (see, for example, \cite{HLW06, KS06}).

For dense graphs, there is a somewhat surprising answer to the first question, in that many of the possible definitions for what it means to be random-like turn out to be equivalent, a fact first observed by Chung, Graham and Wilson~\cite{CGW}, building on earlier work of Thomason~\cite{Th, Th2}. To be more precise, let $H$ be a fixed graph with $r$ vertices and $m$ edges, $0 < p < 1$ a fixed constant and $G$ an $n$-vertex graph. Consider the following properties that $G$ might have:

\begin{enumerate}[leftmargin=1.8cm]

\item[$\mathcal{P}_{H,p}(\epsilon)$:] 
The number of labeled copies of $H$ in $G$ is within $\epsilon n^r$ of $p^m n^r$.

\item[$\mathcal{P}_{H,p}^*(\epsilon)$:] 
For every subset $S \subseteq V(G)$, the number of labeled copies of $H$ in the induced subgraph $G[S]$
is within $\epsilon n^r$ of $p^m |S|^r$.

\end{enumerate}

The property $\mathcal{P}_{H,p}$ asks that the number of copies of $H$ in $G$ is close to what one would expect in the binomial random graph $G(n,p)$, while the hereditary property $\mathcal{P}^*_{H,p}$ asks for a more robust version of this condition, saying that the copies of $H$ are also uniformly distributed in $G$, again just as one would expect in $G(n,p)$. 

A sequence of graphs $(G_n)_{n=1}^\infty$ with $|G_n| = n$ is said to be {\it $p$-quasirandom} if it satisfies the property $\mathcal{P}_{2, p}^*(\epsilon)$ with $\epsilon = o(1)$, where, here and throughout the paper, we write $\mathcal{P}_{r,p}$ and $\mathcal{P}^*_{r,p}$ for $\mathcal{P}_{H,p}$ and $\mathcal{P}^*_{H,p}$ when $H$ is the complete graph $K_r$. That is, a sequence of graphs is $p$-quasirandom if the density of $G_n$ is asymptotic to $p$ and the edges are uniformly distributed across subsets. Of the equivalent formulations discovered by Chung, Graham and Wilson~\cite{CGW}, the most striking is perhaps the following.

\begin{theorem*}[Chung--Graham--Wilson]
For any $\epsilon > 0$, there exists $\delta > 0$ such that
\begin{enumerate}
\item[(a)]
 if a graph satisfies $\mathcal{P}_{2,p}(\delta)$ and $\mathcal{P}_{C_4,p}(\delta)$, then it also satisfies $\mathcal{P}_{2, p}^*(\epsilon)$;
\item[(b)]
if a graph satisfies $\mathcal{P}_{2, p}^*(\delta)$, then it also satisfies $\mathcal{P}_{2,p}(\epsilon)$ and $\mathcal{P}_{C_4,p}(\epsilon)$.
\end{enumerate}
\end{theorem*}

It follows from part (a) that if the number of edges and the number of cycles of length four in a sequence of graphs $(G_n)_{n=1}^\infty$ are both asymptotic to their expected values in the random graph $G(n,p)$, then the sequence is $p$-quasirandom. What is striking about this conclusion is that a purely global property, that of having certain counts of edges and cycles of length four, is sufficient to imply a local property about the distribution of edges on small subsets. Moreover, by part (b), the converse also holds.

A similar result is conjectured~\cite{CFS10, ST04} to hold when $C_4$ is replaced by any bipartite graph $H$ with at least one cycle. Known as the forcing conjecture, this conclusion is now known to hold for a wide range of bipartite graphs, with progress on the conjecture closely paralleling recent progress~\cite{CFS10, CKLL16, KLL16, LSz16, Sz16} on Sidorenko's conjecture. On the other hand, part (a) does not hold when $C_4$ is replaced by a non-bipartite $H$. To see this for triangles, consider the graph $G$ on $n$ vertices consisting of four disjoint sets $V_1, V_2, V_3$ and $V_4$, each of order $n/4$, with $V_1$ and $V_2$ complete, $V_3$ and $V_4$ empty, a complete bipartite graph between $V_3$ and $V_4$ and a random bipartite graph with edge probability $1/2$ between $V_1 \cup V_2$ and $V_3 \cup V_4$. This graph has density $1/2$ and asymptotically $n^3/8$ labeled triangles. However, it is clearly not $1/2$-quasirandom.

A key result about quasirandom graphs, proved by Simonovits and S\'os~\cite{SS97}, says that such counterexamples can be avoided if we strengthen our assumption, asking that $G$ satisfies the hereditary property $\mathcal{P}_{H,p}^*$ rather than just $\mathcal{P}_{H,p}$. That is, a sequence of graphs $(G_n)_{n=1}^\infty$ with $|G_n| = n$ is $p$-quasirandom if every subset $S \subseteq V(G_n)$ contains roughly the same number of copies of $H$ as one would expect to find in $G(n,p)$.

\begin{theorem*}[Simonovits--S\'os]
For any $\epsilon > 0$, there exists $\delta > 0$ such that if a graph satisfies $\mathcal{P}_{H,p}^*(\delta)$, then it also satisfies $\mathcal{P}_{2, p}^*(\epsilon)$.
\end{theorem*}

The original proof of this result makes heavy use of the regularity lemma~\cite{Sz76}, resulting in a bound on $\delta^{-1}$ which is a tower of twos of height polynomial in $\epsilon^{-1}$. The main aim of this paper is to give an alternative proof of this theorem which avoids the use of the regularity lemma, giving a much better bound for $\delta$ in terms of $\epsilon$. This answers a problem raised by Simonovits and S\'os in their paper. In reality, we prove two theorems, the first showing that $\delta$ may be taken to be linear in $\epsilon$ when $H$ is a clique, a result which is clearly optimal (consider the random graph with density $p + \epsilon$), and the second showing that $\delta$ may be taken to be polynomial in $\epsilon$ for general $H$.

\begin{theorem}
\label{complete-graphs}
For every natural number $r \geq 3$, there is a constant $c(r)$ such that if $p$ and $\epsilon$ are constants with $0 < \epsilon, p < 1$ and $n$ is sufficiently large depending on $r$, $p$ and $\epsilon$, then any $n$-vertex graph $G$ that satisfies ${\cal P}^*_{r,p}(\delta)$ with $\delta = c(r) p^{2r^2} \epsilon$ also satisfies ${\cal P}^*_{2,p}(\epsilon)$. %with $\epsilon \leq c(r) p^{-r^2+3} \delta$.
\end{theorem}

\begin{theorem}
\label{graphs-graphs}
For every $0<p<1$ and natural number $r \geq 3$, there are constants $c=c(p,r)$ and $c'(p,r)$ such that 
if $H$ is a graph with $r$ vertices, $0<\epsilon < 1/2$ and a graph $G$ satisfies
${\cal P}^*_{H,p}(\delta)$ with $\delta = c'(p,r) \epsilon^{c(p,r)}$, then it also satisfies ${\cal P}^*_{2,p}(\epsilon)$.
\end{theorem}

We note that another alternative proof for a variant of the Simonovits--S\'os theorem, also avoiding regularity, was found independently by Reiher and Schacht~\cite{RS16}. However, their result uses slightly stronger assumptions and gives weaker quantitative control than ours. As well as being interesting in its own right, our close attention to quantitative aspects was motivated by the possibility of application in extremal combinatorics. Indeed, the best bounds for a number of well-known theorems in this area, including Ramsey's theorem~\cite{C09} and Szemer\'edi's theorem~\cite{G01}, rely crucially on the interplay between different notions of quasirandomness. Our results bring the Simonovits--S\'os theorem into a range where it could also be profitably applied in this manner.

The rest of the paper will be laid out as follows: we study complete graphs in the next section, proving Theorem~\ref{complete-graphs}; we treat the general case in Section~\ref{sec:general}, proving Theorem~\ref{graphs-graphs}; and we conclude with some further remarks and open problems in Section~\ref{sec:conclude}. For the sake of clarity of presentation, we systematically omit floor and ceiling signs whenever they are not crucial.

\section{Complete graphs} \label{sec:complete}

In this section, we prove Theorem~\ref{complete-graphs}. We will need several lemmas about graphs $G$ that satisfy ${\cal P}^*_{r,p}(\delta)$, the first of which estimates the number of $r$-cliques with exactly $i$ vertices in one set $X$ and $r-i$ vertices in another set $Y$. The proof draws on ideas used by Shapira~\cite{S08} when studying a related problem. Here and throughout this section, we will use big $O$ notation, allowing the hidden constants to depend on the clique size $r$ but not on the edge density $p$. In keeping with the statement of Theorem~\ref{complete-graphs}, we will always assume that $n$ is taken sufficiently large.

\begin{lemma}\label{two-sets}
Let $G$ be an $n$-vertex graph that satisfies
${\cal P}^*_{r,p}(\delta)$. Then, for all disjoint subsets $X,Y$ of $V(G)$, the number of labeled $r$-cliques with exactly $i$ vertices in $X$ and $r-i$ vertices in $Y$ deviates from ${r \choose i} p^{{r \choose 2}}|X|^i|Y|^{r-i}$ by at most $O(\delta n^r)$.
\end{lemma}

\begin{proof}
Let $x_i={r \choose i} p^{{r \choose 2}}|X|^i|Y|^{r-i}$ and let $x_i'$ be the number of labeled $r$-cliques with exactly $i$ vertices in $X$ and $r-i$ vertices in $Y$. Pick a random subset $X' \subseteq X$ of order $q|X|$. The expected number of labeled cliques of order $r$ in $X' \cup Y$ is (up to lower order terms)
$\sum_i q^i x'_i$ and this deviates from 
$$p^{{r \choose 2}} (q|X|+|Y|)^r= \sum_i q^ip^{{r \choose 2}}{r \choose i} |X|^i|Y|^{r-i}  = \sum_i q^i x_i$$ 
by at most $O(\delta n^r)$.

For $j = 1, 2, \dots, r+1$, let $q_j=j/(r+1)$ and let $A$ be the $(r+1) \times (r+1)$ matrix with
$a_{ji}=q_j^i$. The matrix $A$ is not singular, since it is a Vandermonde matrix and the $q_j$ are distinct.
Let $a$ be a maximum in absolute value entry of $A^{-1}$, noting that this depends only on $r$.
Let $x=(x_0, x_1, \ldots, x_r)$ and $x'=(x'_0, x'_1, \ldots, x'_r)$.
By the above discussion, we know that the coordinates of the vectors $z=Ax$ and $z'=Ax'$ differ by at most
$O(\delta n^r)$. Since $x-x'=A^{-1}(z-z')$, it follows that the coordinates of the vectors 
$x$ and $x'$ differ by at most $O(r \cdot |a| \cdot \delta n^r)=O(\delta n^r)$, completing the proof.
\end{proof}

We will need a corollary of this lemma saying that for any subset $U \subseteq V(G)$, most $u \in U$ are contained in approximately the same number of $r$-cliques in $U$. The following definition helps capture this condition.

\vspace{3mm}

\noindent
{\bf Definition.} Given a subset $U\subseteq V(G)$ and a vertex $u \in U$, let $c_{U}(u)$ denote the number of $r$-cliques in $U$ containing $u$ and $disc_{U}(u)=\big|c_U(u) - p^{{r \choose 2}} |U|^{r-1}/(r-1)!\big|$.

\begin{corollary}
\label{vertex-discrepancy}
Let $G$ be an $n$-vertex graph that satisfies
${\cal P}^*_{r,p}(\delta)$. Then $\sum_{u \in U} disc_{U}(u) = O(\delta n^r)$. 
\end{corollary}

\begin{proof}
Partition $U$ into two sets $U', U''$ such that $U'$ is the set of all vertices $u \in U$ satisfying
$c_U(u) \geq p^{{r \choose 2}} |U|^{r-1}/(r-1)!$. Then we can write 
$\sum_{u \in U} disc_{U}(u)= \sum_1+\sum_2$, where $\sum_1=\sum_{u \in U'} \big(c_U(u) - p^{{r \choose 2}} |U|^{r-1}/(r-1)!\big)$ and
$\sum_2=-\sum_{u \in U''} \big(c_U(u) - p^{{r \choose 2}} |U|^{r-1}/(r-1)!\big)$. Note that $\sum_1$ can be written as a sum of $r$ terms, each estimating the deviation of $i$ times the number of $r$-cliques in $G[U]$ with exactly $i$ vertices in $U'$ and $r-i$ vertices in $U''$. By Lemma \ref{two-sets}, all these terms are bounded by $O(\delta n^r)$ and, therefore, $\sum_1=O(\delta n^r)$. A similar argument shows that $\sum_2=O(\delta n^r)$.
\end{proof}

\vspace{3mm}

\noindent
{\bf Definition.} Given two vertices $u, v \in V(G)$, let $c(u,v)$ denote the number of subsets $S \subset V(G)$ of order $r-1$ such that
both $\{u\} \cup S$ and $\{v\} \cup S$ form an $r$-clique in $G$ and write 
$$disc(u,v)=\left|c(u,v) - p^{{r \choose 2}} d^{r-1}(v)/(r-1)!\right|,$$
where $d(v)$ is the order of the neighborhood of $v$.

\vspace{3mm}

\noindent
Note that by definition $c(u,v)=c(v,u)$. Therefore, by the triangle inequality, we have
$$disc(u,v)+disc(v,u) \geq  p^{{r \choose 2}} \big|d^{r-1}(v)-d^{r-1}(u)\big|/(r-1)!\,.$$
Using this inequality, we can prove the following lemma which shows that most pairs of vertices in $G$ have comparable degree.

\begin{lemma}\label{degrees}
Let $G$ be an $n$-vertex graph that satisfies
${\cal P}^*_{r,p}(\delta)$. Then 
$$\sum_{u,v \in V(G)} |d^{r-1}(v)-d^{r-1}(u)| = O(p^{-{r \choose 2}}\delta n^{r+1}). $$
\end{lemma}

\begin{proof}
By the above discussion, we have that
$$ \sum_u \sum_{v \neq u} disc(v,u)=\sum_{\{u,v\}}(disc(u,v)+disc(v,u)) \geq  p^{{r \choose 2}} \sum_{\{u,v\}}\big|d^{r-1}(v)-d^{r-1}(u)\big|/(r-1)!\,.$$
Therefore, to prove the statement, it is enough to show that $\sum_{v \neq u} disc(v,u)= O(\delta n^r)$ for each $u$.

Let $U$ be the set of neighbors of $u$ in $G$ and let $W$ be the complement of $U$. Partition $W$ further into $W', W''$, where $W'$ is the set of all vertices $v$ such that  $c(v,u) \geq p^{{r \choose 2}} d^{r-1}(u)/(r-1)!$. Then we can write 
$\sum_{v \neq u} disc(v, u)= \sum_1+\sum_2+\sum_3$, where $\sum_1=\sum_{v \in W'} \big(c(v,u) - p^{{r \choose 2}} d^{r-1}(u)/(r-1)!\big)$,
$\sum_2= -\sum_{v \in W''} \big(c(v,u) - p^{{r \choose 2}} d^{r-1}(u)/(r-1)!\big)$ and 
$\sum_3=\sum_{v \in U} disc_{U}(v)$.  The first (resp., second) sum estimates the deviation of the number of $r$-cliques with one vertex in $W'$ (resp., $W''$) and the remaining $r-1$ vertices in $U$. Thus, by Lemma~\ref{two-sets}, it is bounded by $O(\delta n^r)$.
The third sum is bounded by $O(\delta n^r)$ by Corollary \ref{vertex-discrepancy}. 
\end{proof}

We also need an elementary inequality, which follows as an easy corollary of the next result.

\begin{proposition}
\label{inequality1}
Let $a_1, \dots, a_n$ and $b_1, \ldots, b_n$ be two sets of $n$ non-negative numbers. Then, for any positive integer $s$,
$$\sum_{i,j} |b_j^s-a_i^s|  \geq n \sum_j b_j^s - \sum_j b_j^{s-1} \cdot\sum_i a_i.$$
\end{proposition}  

\begin{proof}
This follows since
\begin{align*}
\sum_{i,j} |b_j^s - a_i^s| & = \sum_{i,j} |b_j - a_i||b_j^{s-1} + \dots + a_i^{s-1}| \geq \sum_{i,j} |b_j - a_i| b_j^{s-1}\\
& \geq \sum_{i,j} (b_j - a_i) b_j^{s-1} = n \sum_j b_j^s - \sum_j b_j^{s-1} \cdot \sum_i a_i,
\end{align*}
where in both inequalities we used the fact that the $a_i$ and $b_j$ are non-negative and in the second inequality we also used the reverse triangle inequality $|x - y| \geq |x| - |y|$.
\end{proof}

\begin{corollary}
\label{inequality2}
Let $a_1, \dots, a_n$ and $b_1, \ldots, b_n$ be two sets of $n$ non-negative numbers. Then, for any positive integer $s$, $\sum_{i,j} |b_j^s-a_i^s|  \geq  \sum_j b_j^{s-1} \cdot(\sum_j b_j -\sum_i a_i).$
\end{corollary}

\begin{proof}
Applying Jensen's inequality twice, first with the function $x^{s/(s-1)}$ and then with the function
$x^{s-1}$, we obtain that 
$\frac{1}{n}\sum_j b_j^s \geq \left(\frac{1}{n}\sum_j b_j^{s-1}\right)^{s/(s-1)}$ and 
$\left(\frac{1}{n}\sum_j b_j^{s-1}\right)^{1/(s-1)} \geq \frac{1}{n}\sum_j b_j$.
Therefore, $\frac{1}{n}\sum_j b_j^s \geq \frac{1}{n}\sum_j b_j^{s-1} \cdot \frac{1}{n}\sum_j b_j$.
Together with Proposition~\ref{inequality1}, this proves the corollary. 
\end{proof}

We now consider a converse to our intended theorem, saying that if a graph satisfies ${\cal P}^*_{2,p}(\gamma)$, then it also satisfies ${\cal P}^*_{r,p}(r^2 \gamma)$. Versions of this counting lemma already appear in the literature, for example, in Section 10.5 of Lov\'asz' book on graph limits~\cite{Lo}. However, because this result is central to our estimates and not as well known as it should be, we include the proof.

\begin{lemma} \label{lem:counting}
If a graph $G$ satisfies ${\cal P}^*_{2,p}(\gamma)$, then it also satisfies ${\cal P}^*_{r,p}(r^2 \gamma)$.
\end{lemma}

\begin{proof}
Given $S, T \subseteq V(G)$, write $e(S, T) = \sum_{s \in S, t \in T} 1_G(s,t)$, where $1_G$ is the indicator function for edges of $G$. In particular, when $S = T$, this counts the number of labeled edges in $S$. By assumption, $e(S, S) = p |S|^2 \pm \gamma n^2$ for all $S \subseteq V(G)$. Therefore, using the identity 
\[2e(S, T) = e(S\cup T, S \cup T) + e(S\cap T, S \cap T) - e(S \setminus T, S \setminus T) - e(T \setminus S, T \setminus S),\]
we see that $e(S, T) = p |S||T| \pm 2 \gamma n^2$ for all $S, T \subseteq V(G)$. Rewriting this conclusion, we see that
\[|\sum_{s \in S, t \in T} (1_G(s,t) - p)| \leq 2 \gamma n^2\]
for all $S, T \subseteq V(G)$.
In turn, this implies that for any functions $u, v: V(G) \rightarrow [0,1]$,
\[|\sum_{x, y \in V(G)} (1_G(x,y) - p) u(x) v(y)| \leq 2 \gamma n^2.\]
Indeed, since the function we wish to estimate is linear in $u(x)$ and $v(y)$ for each $x$ and $y$, the value of the function is maximised when $u$ and $v$ take values in $\{0,1\}$. In this case, $u$ and $v$ correspond to indicator functions, so the inequality reduces to the previous special case.

For ease of notation, we spell out the rest of the proof for the case of triangles. By telescoping, the deviation between the number of labeled triangles in a set $S \subseteq V(G)$ and its expected value is
\begin{align*}
\sum_{x, y, z \in S} (1_G(x,& y) 1_G(y,z) 1_G(z,x) - p^3) = \sum_{x,y,z \in S} (1_G(x,y) - p) 1_G(y,z) 1_G(z,x) + \\
& \sum_{x,y,z \in S} p (1_G(y,z) - p) 1_G(z,x) + \sum_{x,y,z \in S} p^2 (1_G(z,x) - p).
\end{align*}
Each term on the right-hand side of this equation may be written as a sum over terms of the form $\sum_{x, y \in V(G)} (1_G(x,y) - p) u(x) v(y)$ for some appropriate $u$ and $v$, thus implying that the deviation we are interested in is at most $6 \gamma n^3$. In general, we will be telescoping over $\binom{r}{2}$ terms, one for each edge in $K_r$, so the resulting deviation is $\binom{r}{2} 2 \gamma n^r \leq r^2 \gamma n^r$, as required.
\end{proof}

We will also use some simple ingredients from other papers. The first, taken from a paper of Erd\H{o}s, Goldberg, Pach and Spencer~\cite{EGPS}, says that if an $n$-vertex graph contains a set which deviates from the expected density, then there is also a set of order $n/2$ which deviates from this density.

\begin{lemma} \label{lem:EGPS}
Let $G$ be an $n$-vertex graph of density $q$. If there is a subset $S \subseteq V(G)$ for which $|e(S) - q\binom{|S|}{2}| \geq D$, then there exists a set $S'$ of order $n/2$ such that $|e(S') - q\binom{|S'|}{2}| \geq \left(\frac{1}{4}  + o(1)\right) D$.
\end{lemma}

We also need a special case of the Kruskal--Katona theorem~\cite{K68, Kr63} giving an upper bound for the number of $r$-cliques in a graph with given density. The result we use is given as Exercise 31b in Chapter 13 of Lov\'asz' problem book~\cite{Lo07}. Here the binomial coefficient $\binom{x}{r}$ is extended to all real $x$ in the obvious way.

\begin{lemma} \label{lem:KK}
Let $r \geq 3$ be an integer and $x \geq r$ a real number. Then a graph with exactly $\binom{x}{2}$ edges contains at most $\binom{x}{r}$ cliques of order $r$.
\end{lemma}

Finally, we need the standard Azuma--Hoeffding inequality, which we apply in the following form (see Corollary 2.27 in~\cite{JLR}).

\begin{lemma} \label{Azuma}
Given positive real numbers $\lambda, c_1, \dots, c_k$, let $f: \{0,1\}^k \rightarrow \mathbb{R}$ be a function satisfying the following Lipschitz condition: whenever two vectors $z, z' \in \{0,1\}^k$ differ only in the $i$th coordinate, $|f(z) - f(z')| \leq c_i$. Then, if $X_1, \dots, X_k$ are independent random variables, each taking values in $\{0,1\}$, the random variable $Y = f(X_1, \dots, X_k)$ satisfies
\[\mathbb{P}[|Y - \mathbb{E}[Y]| \geq \lambda] \leq 2 \exp\left\{- \frac{\lambda^2}{2\sum_i c_i^2}\right\}.\] 
\end{lemma}

We are now ready to prove Theorem~\ref{complete-graphs}. We will do this by showing that if a sufficiently large graph $G$ satisfies
${\cal P}^*_{r,p}(\delta)$, then it also satisfies ${\cal P}^*_{2,p}(\epsilon)$ with $\epsilon =O(p^{-2 r^2} \delta)$.

\vspace{3mm}

\noindent
{\bf Proof of Theorem \ref{complete-graphs}:}\, Let $G$ be an $n$-vertex graph satisfying the assertion of the theorem and let 
$q=e(G)/{n \choose 2}$ be the edge density of $G$. By the Kruskal--Katona theorem, Lemma~\ref{lem:KK}, $G$ contains $O(q^{r/2} n^r)$ labeled $r$-cliques. Since the number of labeled $r$-cliques in $G$ is also at least $\frac{1}{2} p^{\binom{r}{2}} n^r$, this implies that $q = \Omega(p^{r-1})$.

Let $\gamma =c q^{-(r-2)} p^{-{r \choose 2}} \delta$ for some constant $c$ which we choose later. If $G$ does not satisfy ${\cal P}^*_{2,q}(\gamma)$, then it contains a subset of vertices $S$ such that $e(S)$ deviates from $q{|S| \choose 2}$ by at least $\gamma n^2$. Using Lemma~\ref{lem:EGPS}, we can assume that $S$ has order $n/2$, allowing for the possibility that $\gamma$ may change by a small constant factor. 

Let $A$ be a random subset of $G$, obtained by choosing every vertex independently with probability $1/2$. Let $X=S \cap A$ and let $Y$ be a random subset of $G \setminus A$, obtained by further choosing every vertex with probability $1/2$. By linearity, the expected number of vertices in $X$ is $n/4$
and the expected number of edges in $X$ is $e(S)/4$. Moreover, both of these quantities are highly concentrated by the
Azuma--Hoeffding inequality. Indeed, changing the choice for one vertex can change the number of vertices in $X$ by at most one and the number of edges by at most $n$, so the sum of squares of these changes is at most $O(n)$ for the number of vertices and $O(n^3)$ for the number of edges.
These are much smaller than the square of the corresponding expectations.
Similarly, the expected number of vertices and edges in $Y$ are $n/4$ and 
$e(G)/16$, respectively, and they are also both concentrated. Therefore, we can find two disjoint subsets $X$ and $Y$, each of order
$n/4+o(n)$, such that $e(X)=e(S)/4+o(n^2)$ and $e(Y)=e(G)/16+o(n^2)$. Thus, by the discussion above, we have $\big| e(X)-e(Y)\big| = \Omega(\gamma n^2)$ and we can delete $o(n)$ vertices to make the orders of $X$ and $Y$ equal without changing this inequality. 
 
Let $U=X \cup Y$ and let $H=G[U]$ be the subgraph of $G$ induced by $U$. Without loss of generality, we will assume that $e(X) \geq e(Y)$. 
In particular, the edge density in $X$ is at least $q$. 
Since 
$\sum_{x \in X} d_H(x)=2e(X)+e(X,Y)$  and
$\sum_{y \in Y} d_H(y)=2e(Y)+e(X,Y)$, we have 
$\sum_{x \in X} d_H(x)-\sum_{y \in Y} d_H(y) =  \Omega(\gamma n^2)$.
Therefore, using Corollary \ref{inequality2} (with $s=r-1$) and $|X|=n/4+o(n)$, we deduce that
\begin{eqnarray*}
\sum_{u,v \in H} \big|d_H^{r-1}(u)-d_H^{r-1}(v)\big| &\geq& \sum_{x \in X, y \in Y} \big|d_H^{r-1}(x)-d_H^{r-1}(y)\big|
\geq \sum_{x \in X} d_H^{r-2}(x) \left(\sum_{x \in X} d_H(x)- \sum_{y \in Y} d_H(y) \right)\\
&\geq& |X| \left( \sum_{x \in X} d_H(x)/|X| \right)^{r-2}  \Omega(\gamma n^2)=\Omega\big(q^{r-2} \gamma n^{r+1}\big).
\end{eqnarray*}
For a sufficiently large constant $c$ (in the definition of $\gamma$), this contradicts Lemma \ref{degrees} and implies that 
$G$ satisfies ${\cal P}^*_{2,q}(\gamma)$. Finally, by Lemma~\ref{lem:counting}, we have that any graph satisfying
${\cal P}^*_{2,q}(\gamma)$ also satisfies ${\cal P}^*_{r,q}(r^2\gamma)$. Therefore, the number of labeled $r$-cliques in $G$ deviates from
$q^{r \choose 2}n^r$ by at most $r^2 \gamma n^r$.  On the other hand, the difference between
$(p \pm \epsilon)^{r \choose 2}$ and $p^{r \choose 2}$ has order of magnitude $\Omega\big(p^{{r \choose 2}-1}\epsilon\big)$. Thus, if $q$ differs from $p$ by $\epsilon=c' p^{-{r \choose 2}+1}\gamma$ for sufficiently large $c'$, we obtain the wrong count of $r$-cliques in $G$, contradicting ${\cal P}^*_{r,p}(\delta)$. Therefore, $G$ must satisfy ${\cal P}^*_{2,p}(\epsilon)$. Since $q = \Omega(p^{r-1})$, $\epsilon = c c' q^{-(r-2)} p^{-2\binom{r}{2} + 1} \delta = O\big(p^{-2 r^2} \delta\big)$, completing the proof.
\hfill $\Box$

\section{General graphs} \label{sec:general}

In this section, we prove Theorem~\ref{graphs-graphs}. We will assume throughout that $H$ does not have isolated vertices, as deleting such a vertex from $H$ simply scales the number of labeled copies in $S$ by a factor of $|S|-r+1$. 

We say that an $n$-vertex graph $G$ has property ${\cal Q}_{H,p}(\epsilon)$ if, for every $r$ disjoint subsets $V_1,\ldots,V_r \subseteq V(G)$, the number of labeled copies of $H$ with one vertex in each $V_i$ is $p^{m}r!\prod_{i=1}^r |V_i| \pm \epsilon n^r$. 
In other words, property ${\cal Q}_{H,p}(\epsilon)$ says that if we average over all possible permutations $\pi$ of $[r]$ the number of copies of $H$ with the copy of vertex $i$ in $V_{\pi(i)}$, the result is at most $\epsilon n^r/r!$ from $p^m\prod_{i=1}^r |V_i|$.

For a subset $U \subseteq V(G)$, let $N_H(U)$ denote the number of labeled copies of $H$ in $G$ whose vertices lie in $U$. Let $N_H(V_1,\ldots,V_r)$ denote the number of labeled copies of $H$ in $G$ with one vertex in each $V_i$. For $S \subseteq [r]$, let $U_S=\bigcup_{i \in S} V_i$. By the inclusion-exclusion principle, we have 
$$N_H(V_1,\ldots,V_r) = \sum_{S \subseteq [r]} (-1)^{r-|S|} N_H(U_S).$$ 
If $G$ has property ${\cal P}^*_{H,p}(\epsilon)$, it follows that $N_H(U_S)$ is within $\epsilon n^r$ of $p^m|U_S|^r$.  Applying this to each of the $2^r - 1$ choices of $S$, we get that $N_H(V_1,\ldots,V_r)$ is within $(2^r-1)\epsilon n^r$ of $p^m r! \prod_{i=1}^r |V_i|$. Hence, we have the following lemma. 

\begin{lemma}\label{easy}
If $G$ satisfies ${\cal P}^*_{H,p}(\epsilon)$, then it also satisfies ${\cal Q}_{H,p}((2^r-1)\epsilon)$.
\end{lemma}

We remark that the property studied by Reiher and Schacht \cite{RS16} is a variant of ${\cal Q}_{H,p}$.  We say that a graph $G$ on $n$ vertices has property $\mathcal{R}_{H,p}(\epsilon)$ if for any $r$ disjoint vertex subsets $V_1,\ldots,V_r$ of $G$ and every one-to-one mapping $\pi:V(H) \rightarrow [r]$, the number of copies of $H$ where the image of $v$ is in $V_{\pi(v)}$ for each vertex $v$ of $H$ is within $\epsilon n^r$ of $p^m\prod_{i=1}^r |V_i|$. Note that $\mathcal{R}$ is a stronger property than $\mathcal{Q}$, since if a graph satisfies $\mathcal{R}_{H,p}(\epsilon)$, then it also satisfies $\mathcal{Q}_{H,p}(r! \epsilon)$. It remains an open problem to find a simple proof (i.e., without going through the methods developed here or through regularity methods) that shows the other direction, that  $\mathcal{Q}$ implies $\mathcal{R}$.

We say that a pair $(A,B)$ of vertex subsets of a graph $G$ is {\it lower-$(q,\epsilon)$-regular} if, for all  $A' \subseteq A$ and $B' \subseteq B$, 
$$e(A',B') \geq q|A'||B'|-\epsilon|A||B|.$$ 
That is, the density between all pairs of large subsets is at least $q$, up to an error depending on $\epsilon$. As in the proof of Lemma~\ref{lem:counting}, this is equivalent to saying that for all functions $u:A \rightarrow [0,1]$ and $v:B \rightarrow [0,1]$, 
$$\sum_{a \in A,b \in B} 1_G(a,b) u(a)v(b) \geq q\sum_{a \in A,b \in B} u(a)v(b) - \epsilon|A||B|.$$ 

Similarly, we say that $(A,B)$ is  {\it upper-$(q,\epsilon)$-regular} if, for all subsets $A' \subseteq A$ and $B' \subseteq B$, 
$$e(A',B') \leq q|A'||B'|+\epsilon|A||B|.$$ 
We note that if a pair of subsets is both lower-$(q,\epsilon)$-regular and upper-$(q,\epsilon)$-regular, it satisfies a notion of regularity introduced by Lov\'asz and Szegedy~\cite{LS} (though equivalent up to a polynomial change in $\epsilon$ to the original notion of regularity introduced by Szemer\'edi~\cite{Sz76}). 

The following counting lemma gives a lower bound for the number of copies of a graph $H$ between a collection of sets with lower-regular pairs. We omit the proof, which follows by the same telescoping argument used for Lemma~\ref{lem:counting}.

\begin{lemma}\label{countlemma}
Let $H$ be a graph on vertex set $\{1,2,\ldots,r\}$. Let $G$ be a graph with vertex subsets $V_1,\ldots,V_r$ such that $(V_i,V_j)$ is lower-$(p_{ij},\epsilon)$-regular for each edge $(i,j)$ of $H$. Then the number of homomorphisms from $H$ to $G$ with the copy of vertex $i$ in $V_i$ is at least $$\left(\prod_{(i,j) \in E(H)}p_{ij} -e(H)\epsilon\right)\prod_{i=1}^r |V_i|.$$
\end{lemma}

Note that a similar lemma also holds if lower is replaced by upper and $-$ by $+$. It is also worth noting that we have not insisted that the vertex sets $V_1, \dots, V_r$ be disjoint. In particular, we may take $V_1 = \dots = V_r$ to obtain a non-partite version of the lemma.

The next lemma shows that lower regularity implies upper regularity and vice versa.

\begin{lemma}\label{lowervsupper}
If a pair $(A,B)$ of vertex subsets of a graph is not lower-$(d(A,B),\epsilon)$-regular, then it is also not upper-$(d(A,B),\epsilon/2)$-regular. The same holds if lower and upper are switched.
\end{lemma}
\begin{proof}
By assumption, there are subsets $A' \subset A$ and $B' \subset B$ such that $e(A',B')-d(A,B)|A'||B'|<-\epsilon|A||B|$. As \begin{eqnarray*}e(A',B')+e(A \setminus A',B)+e(A',B \setminus B') & = & e(A,B)\\ & = & d(A, B) |A||B|\\ & = & d(A,B)|A'||B'|+d(A,B)|A \setminus A'||B|+d(A,B)|A'||B \setminus B'|,\end{eqnarray*} it follows that at least one of the pairs $(A \setminus A',B)$ and $(A',B \setminus B')$ demonstrates that $(A,B)$ is not upper-$(d(A,B),\epsilon/2)$-regular. The proof when lower and upper are switched is the same.
\end{proof}

We also need a simple lemma saying that we can always find a pair of subsets of equal size which bear witness to irregularity. 

\begin{lemma}\label{lowup}
Suppose $(A,B)$ is a pair of vertex subsets of a graph which is not upper-$(q,\gamma)$-regular. Then there are subsets $A' \subseteq A$ and $B' \subseteq B$ with $|A'|=|B'|$ and $d(A',B') \geq  q+\gamma\min\{\frac{|A|}{|A'|}, \frac{|B|}{|B'|}\}$. The same holds with upper replaced by lower, $+$ by $-$ and the inequality reversed. 
\end{lemma}

\begin{proof}
As $(A,B)$ is not upper-$(q,\gamma)$-regular, there are subsets $A_0 \subseteq A$ and $B_0 \subseteq B$ such that 
$e(A_0,B_0) \geq q|A_0||B_0|+\gamma|A||B|$. Without loss of generality, we may assume that $|A_0| \leq |B_0|$. Let $B' \subseteq B_0$ be the subset of $|A_0|$ vertices with the most neighbors in $A_0$ and $A'=A_0$. Then 
$$d(A',B') \geq d(A_0,B_0) = \frac{e(A_0,B_0)}{|A_0||B_0|} \geq q+\gamma\frac{|A||B|}{|A_0||B_0|} \geq q+\gamma\frac{|A|}{|A'|},$$
as required. 
\end{proof}

The next lemma shows that if a graph satisfies ${\cal Q}_{H,p}$ but there are sets $A$ and $B$ for which the density $d(A, B)$ deviates significantly from the expected density, then there are sets $A'$ and $B'$ such that the density $d(A', B')$ deviates by even more. This observation will allow us to run a density-increment argument in the proof of Theorem~\ref{graphs-graphs}.

\begin{lemma} \label{keylem} Let $H$ be a graph with $r$ vertices and $m$ edges. Suppose $G$ is a graph on $n$ vertices that satisfies ${\cal Q}_{H,p}(\delta)$ and has disjoint subsets $A$ and $B$ with $|A|=|B|$ such that $\delta \leq \frac{1}{4r}\alpha mp^m r^{-r} (|A|/n)^r$ and $d(A,B) \geq (1+\alpha)p$ or $d(A,B) \leq (1-\alpha)p$ with $\alpha \leq \frac{1}{16mr}$. Then there are also disjoint subsets $A'$ and $B'$ with $|A'|=|B'|$ and $d(A',B') \geq (1+(1+\beta)\alpha)p$ or $d(A',B') \leq (1-(1+\beta)\alpha)p$ where  $\beta = \frac{p^{m-1}}{4r^3}\frac{|A|}{|A'|}$.  
\end{lemma}
\begin{proof}
Suppose that we are in the case where $d(A,B) \geq  (1+\alpha)p$.
Let $q=(1-\alpha)p$ and $\upsilon = p^m\alpha/(4r)$. Take an arbitrary equitable partition of $A$ into $\lfloor r/2 \rfloor$ subsets  $A_1,\ldots,A_{\lfloor r/2 \rfloor}$ and $B$ into $\lceil r/2 \rceil$ subsets $A_{\lfloor r/2 \rfloor+1},\ldots,A_r$. 

First suppose that there is a pair $(A_i,A_j)$ with $1 \leq i \leq \lfloor r/2 \rfloor < j \leq r$ which is not lower-$(d(A,B),\upsilon)$-regular. As $A_i \subset A$ and $A_j \subset B$, $(A,B)$ is not lower-$(d(A,B),\upsilon')$-regular, where $\upsilon'=\upsilon\frac{|A_i|||A_j|}{|A||B|} \geq \frac{2}{r^2}\upsilon$. By Lemma \ref{lowervsupper}, $(A,B)$ is also not upper-$(d(A,B),\upsilon'/2)$-regular. Therefore, by Lemma \ref{lowup}, there are subsets $A' \subseteq A$ and $B' \subseteq B$ such that $|A'|=|B'|$ and $d(A',B') \geq d(A,B)+\frac{\upsilon'}{2}\frac{|A|}{|A'|} \geq d(A,B)+\frac{\upsilon}{r^2}\frac{|A|}{|A'|}$. Hence, since $\frac{\upsilon}{r^2}\frac{|A|}{|A'|} = \alpha \beta p$, we may suppose all pairs $(A_i,A_j)$ with $1 \leq i \leq \lfloor r/2 \rfloor < j \leq r$ are lower-$(d(A,B),\upsilon)$-regular.

Now suppose there is a pair $(A_i,A_j)$ with $1 < i <  j \leq \lfloor r/2 \rfloor$ which is not lower-$(q,\upsilon)$-regular.  By Lemma \ref{lowup}, there are subsets $A' \subseteq A_i$ and $B' \subseteq A_j$ with $|A'|=|B'|$ and $d(A',B') \leq  q-\upsilon\min\{\frac{|A_i|}{|A'|},\frac{|A_j|}{|B'|}\} \leq q-\frac{\upsilon}{r}\frac{|A|}{|A'|} \leq (1-(1+\beta)\alpha)p$. Hence, we may suppose all pairs $(A_i,A_j)$ with $1 < i <  j \leq \lfloor r/2 \rfloor$ are lower-$(q,\upsilon)$-regular. Similarly, we may suppose all pairs $(A_i,A_j)$ with $ \lfloor r/2 \rfloor < i <  j \leq r$ are lower-$(q,\upsilon)$-regular.

Consider a bijection $\phi:V(H) \rightarrow [r]$. Let $m_1$ denote the number of edges $(v,w)$ of $H$ for which $\phi(v),\phi(w) \leq   \lfloor r/2 \rfloor $ or $\phi(v),\phi(w) >  \lfloor r/2 \rfloor$ and $m_2$ denote the number of edges $(v,w)$ of $H$ for which $\phi(v) \leq   \lfloor r/2 \rfloor < \phi(w)$, so $m=m_1+m_2$. By Lemma \ref{countlemma}, the number of copies of $H$ where $v$ maps to $A_{\phi(v)}$ for each vertex $v \in H$ is at least

\begin{align}
\label{eqn:1}
\begin{split}
\left(q^{m_1}d(A,B)^{m_2}-m\upsilon\right)\prod_{i=1}^r |A_i| & \geq   \left((1-\alpha)^{m_1}(1+\alpha)^{m_2}p^m-m\upsilon\right)\prod_{i=1}^r |A_i| \\ & \geq  \left((1-\alpha m_1)(1+\alpha m_2)p^m-m\upsilon\right)\prod_{i=1}^r |A_i| \\ & =  \left((1-\alpha m_1 + \alpha m_2- \alpha^2 m_1m_2)p^m-m\upsilon\right)\prod_{i=1}^r |A_i| \\ & \geq 
 \left((1- \alpha m_1 + \alpha m_2- \alpha m/64r )p^m-m\upsilon\right)\prod_{i=1}^r |A_i|, \end{split}
\end{align}
where in the last inequality we used $m_1m_2 \leq m^2/4$, which follows from $m_1+m_2=m$ and the AM--GM inequality, and $\alpha \leq 1/16mr$. 

We average this lower bound over all choices of $\phi$. Each edge $(v,w)$ maps to a pair $(i,j)$ with $i<j$ and the probability this pair satisfies $i \leq \lfloor r/2 \rfloor < j$ is $\lfloor r/2 \rfloor \lceil r/2 \rceil /{r \choose 2} \geq \frac{1}{4}(r^2-1)/{r \choose 2}=\frac{1}{2}\left(1+\frac{1}{r}\right)$. Thus, by linearity of expectation, $\mathbb{E}[m_2-m_1] \geq m/r$. Hence, the average value of (\ref{eqn:1}) is at least 
$$ \left(\left(1+\alpha\mathbb{E}[m_2-m_1]-\alpha m/64r\right)p^m-m\upsilon\right)\prod_{i=1}^r |A_i|$$ 
which is at least 
$$\left(\left(1+\alpha m/r-\alpha m/64r\right)p^m-m\upsilon\right)\prod_{i=1}^r |A_i| \geq 
\left(\left(1+\frac{63}{64}\alpha m/r\right)p^m-m\upsilon\right)\prod_{i=1}^r |A_i|.$$
This in turn is equal to $$p^m\prod_{i=1}^r |A_i| + \frac{47}{64r}\alpha mp^m \prod_{i=1}^r |A_i| > p^m\prod_{i=1}^r |A_i| + \frac{1}{4r}\alpha mp^m r^{-r} (|A|/n)^r n^r \geq  p^m\prod_{i=1}^r |A_i| +\delta n^r,$$
contradicting the assumption that $G$ satisfies ${\cal Q}_{H,p}(\delta)$ and completing the proof in this case. The case $d(A,B) \leq  (1-\alpha)p$ follows similarly.%except in \eqref{eqn:1} we need to use an inequality which allows us to approximate $(1+\alpha)^{m_2}$ by $1+m_2\alpha$ up to a small enough error using the upper bound on $\alpha$. 
\end{proof}

A single vertex subset $U$ of a graph is called \emph{$\epsilon$-regular} if, for each pair of subsets $A',B' \subset U$, 
we have $\left|e(A',B')-d(U)|A'||B'|\right| \leq \epsilon|U|^2$. In~\cite{CF12}, the first two authors proved the following lemma with $\delta^{-1}$ having a double-exponential dependence on $\epsilon^{-1}$. 

\begin{lemma}\label{singlereg}
For each $\epsilon>0$, there is $\delta>0$ such that every graph on $n$ vertices has an $\epsilon$-regular subset on at least $\delta n$ vertices. 
\end{lemma}

We are now ready to prove Theorem~\ref{graphs-graphs}. That is, we will show that if $G$ satisfies
${\cal P}^*_{H,p}(\delta)$ with $\delta = c'(p,r) \epsilon^{c(p,r)}$, then it also satisfies ${\cal P}^*_{2,p}(\epsilon)$, where $r$ is the number of vertices in $H$.

\vspace{3mm}

\noindent
{\bf Proof of Theorem \ref{graphs-graphs}:} Suppose for contradiction that $G$ is a graph on $n$ vertices which 
satisfies property ${\cal P}^*_{H,p}(\delta)$, but does not satisfy property ${\cal P}^*_{2,p}(\epsilon)$, where $\delta=c'(p,r)\epsilon^{c(p,r)}$ with $c(p,r)=10r^4p^{1-m}$ and $c'(p,r) > 0$ will be chosen sufficiently small depending only on $p$ and $r$. By Lemma \ref{easy}, $G$ also has property ${\cal Q}_{H,p}((2^r-1)\delta)$. As $G$ does not satisfy property ${\cal P}^*_{2,p}(\epsilon)$, there is $S \subseteq V(G)$  with $\left|2e(S)-p|S|^2\right|> \epsilon n^2$. Averaging over all equitable partitions $S=A \cup B$, we obtain that there are disjoint vertex subsets $A_0$ and $B_0$ with $|A_0|=|B_0|$ and $|e(A_0,B_0)-p|A_0||B_0||> \epsilon n^2/4$. We have $|A_0|^2=|A_0||B_0|>\epsilon n^2/4$, so that $|A_0| \geq \epsilon^{1/2}n/2$ and also $|d(A_0,B_0)-p|>\epsilon n^2/(4|A_0||B_0|) \geq \epsilon$. 

We repeatedly apply Lemma \ref{keylem}, starting with $A_0$ and $B_0$, until we arrive at a pair of subsets $(A',B')$ with $|A'|=|B'|$ and $d(A',B') = (1 + \alpha)p$ or $d(A',B') = (1-\alpha)p$, for some $\alpha \geq \frac{1}{16mr}$. As we will see below, the sets defined during this process will always be sufficiently large that we may continue applying Lemma~\ref{keylem} until this happens. Having already found disjoint sets $A_i$ and $B_i$ with $\left |\frac{d(A_i,B_i)}{p} - 1 \right | : = \alpha_i$ (so, in particular, $\alpha_0 > \epsilon/p$), we apply Lemma \ref{keylem} to obtain disjoint sets $A_i'$ and $B_i'$ with $|A_i'|=|B_i'|$ and $|\frac{d(A_i',B_i')}{p}-1| \geq (1+\beta_i)\alpha_i$, where $\beta_i = \frac{p^{m-1}}{4r^3}\frac{|A_i|}{|A_i'|}$. Let $A_{i+1}$ and $B_{i+1}$ be disjoint sets with $|A_{i+1}|$ as large as possible such that $|A_{i+1}|=|B_{i+1}| \geq |A_i'|$ and $\alpha_{i+1}:=\left |\frac{d(A_{i+1},B_{i+1})}{p}-1\right |  \geq \left|\frac{d(A_i',B_i')}{p}-1 \right |$ is as large as possible (such sets exist because $A_i'$ and $B_i'$ have the desired properties).

Let $a_i=|A_{i-1}|/|A_i|$. Note that $a_i \geq 1$, as otherwise we would have taken $A_i$ and $B_i$ for $A_{i-1}$ and $B_{i-1}$, respectively. Let $\gamma=\frac{p^{m-1}}{4r^3}$, so that $\alpha_{i+1} \geq (1+\gamma a_{i+1})\alpha_i$. Let $i_0$ be the last $i$ for which we obtain an $A_i$ and $B_i$, so that either $\alpha_{i_0}  \geq \frac{1}{16mr}$ or the sets $A'=A_{i_0}$ and $B'=B_{i_0}$ are too small for the hypotheses of Lemma \ref{keylem} to hold. We have 
$$\alpha_{i_0}=\alpha_0\prod_{j=1}^{i_0} \frac{\alpha_j}{\alpha_{j-1}} \geq \alpha_0 \prod_{j=1}^{i_0} (1+\gamma a_j).$$
As also $\alpha_{i_0}  \leq 1/p$ and $\alpha_0 \geq \epsilon/p$, we get $\prod_{j=1}^{i_0} (1+\gamma a_j) \leq \frac{1}{\epsilon}$. Given this inequality, the maximum of $\prod_{j=1}^{i_0} a_j$ is attained when all the $a_j$ are equal; call this equal value $a$. Thus, we are interested in the maximum of $a^{i_0}$ given that $(1+\gamma a)^{i_0} \leq \frac{1}{\epsilon}$. This is equivalent to maximizing $\frac{\ln a}{\ln (1+\gamma a)}$. Let $x=\gamma a$. For $x \leq 1.5$, we have $\ln (1+x) \geq x/2$ and so $\frac{\ln a}{\ln (1+\gamma a)} \leq 2\frac{\ln a}{x} = \frac{2}{\gamma}\frac{\ln a}{a} \leq \frac{2}{e\gamma}$. For $x > 1.5$, we have $\frac{\ln a}{\ln (1+\gamma a)}\leq \frac{\ln a}{\ln (\gamma a)}=1+\frac{\ln \gamma^{-1}}{\ln x} \leq 3\ln \gamma^{-1}$. As we may assume $r \geq 3$ and $m \geq 2$, $\gamma$ is small enough that the first bound is larger, so we have $$|A'|=|A_0|/\prod_{j=1}^{i_0} a_j \geq \epsilon^{2/(e\gamma)}|A_0| \geq \frac{1}{2}\epsilon^{1/2 + 2/(e\gamma)}n.$$ 
By the choice of $c(p,r)$, $|A'|=|A_{i_0}|$ is large enough that Lemma \ref{keylem} still applies at the next step if $\alpha:=\alpha_{i_0} \leq \frac{1}{16mr}$. Therefore, we must have $\alpha >  \frac{1}{16mr}$.

The sets $A'$ and $B'$ have order at least $\epsilon^{5r^3p^{1-m}}n$. Suppose  that $d(A',B') = (1+\alpha)p$ with $\alpha > \frac{1}{16mr}$. The other case when $d(A',B') < (1 - \alpha)p$ is handled similarly. Let $C$ be the subset of $A'$ with at least $(1+\alpha/2)p|B'|$ neighbors in $B'$, so $|C| \geq |A'|\alpha p/2$. Let $\eta:=10^{-5}\alpha^2p^{2m}m^{-2}$ and $\kappa=2m\eta p^{1-m}$. Apply Lemma \ref{singlereg} to the subgraph of $G$ induced by $C$. We get a subset $C_1 \subset C$  which is $\eta$-regular with $|C_1| \geq \tau|C|$, where $\tau$ only depends on $p$ and $r$. If 
$d(C_1)<p-\kappa$, then by the counting lemma, Lemma \ref{countlemma}, the number of homomorphisms from $H$ to $C_1$ is at most  $$(p-\kappa)^{m}|C_1|^r+m \eta|C_1|^r \leq (1-\frac{\kappa}{p})p^m|C_1|^r +m \eta|C|^r = p^m|C_1|^r-m\eta|C_1|^r \leq 
p^m|C_1|^r - a(p,r) \epsilon^{8r^4p^{1-m}}n^r$$  
for an appropriate constant $a(p,r)$. This contradicts ${\cal P}^*_{H,p}(\delta)$ when $c'(p,r)$ is sufficiently small. Hence, $d(C_1) \geq p-\kappa$. Let $D$ be the subset of $B'$ with at least $(1+\alpha/4)p|C_1|$ neighbors in $C_1$, so $|D| \geq |B'|\alpha p/4$. Let $D_1 \subset D$ be  a subset of order $\frac{\alpha p^{m-1}}{40r}|C_1|$. 

The rest of the proof is in showing that $C_1 \cup D_1$ violates the property $\mathcal{P}^*_{H,p}(\delta)$ as this subset contains too many labeled copies of $H$. Indeed, by Lemma \ref{countlemma}, the number of homomorphisms from $H$ to $C_1$ is at least  $(p-\kappa)^m|C_1|^r-m\eta|C_1|^r  \geq p^m|C_1|^r-(mp^{m-1}\kappa+m\eta)|C_1|^r$. The number of homomorphisms from $H$ to $C_1$ which fail to be copies of $H$ is at most the number of non-injective mappings from $H$ to $C_1$, which is less than $r^2|C_1|^{r-1}$. Thus, we get at least $p^m|C_1|^r-(mp^{m-1}\kappa+m\eta+r^2|C_1|^{-1})|C_1|^r$ copies of $H$ in $C_1$. 

We next give a lower bound on the number of copies of $H$ in $C_1 \cup D_1$ with the copy of vertex $i$ in $D_1$ and the remaining $r-1$ vertices in $C_1$. Suppose vertex $i$ has degree $t$ in $H$. We have $t\geq 1$ as $H$ has no isolated vertices. Fix a vertex $v \in D_1$ to map $i$ to, and let $C_2$ be the neighborhood of $i$ in $C_1$, so $|C_2| \geq (1+\alpha/4)p|C_1|$.  Since $|C_2| \geq p|C_1|$ and $C_1$ is $\eta$-regular, we get that each of the pairs  $(C_1,C_1)$, $(C_1,C_2)$, $(C_2,C_2)$ is $2 p^{-2}\eta$-regular. The number of homomorphisms of $H$ with the copy of vertex $i$ mapping to $v$ and the remaining vertices of $H$ mapping to $C_1$ (so each of the $t$ neighbors of $i$ has to map to $C_2$) is, by Lemma \ref{countlemma}, at least 
\begin{eqnarray*}\left((p-\kappa)^{m-t}-(m-t)2 p^{-2}\eta \right)|C_1|^{r-t-1}|C_2|^t & \geq & \left(p^{m-t}-(m-t)\kappa p^{m-t-1}-(m-t)2 p^{-2}\eta\right)|C_1|^{r-t-1}|C_2|^t \\ & \geq & (p^{m-t}-\alpha p^{m-t}/8) |C_1|^{r-t-1}|C_2|^t  \\ & \geq & 
(p^{m-1}-\alpha p^{m-1}/8) |C_1|^{r-2}|C_2|^1
 \\ & \geq & (p^{m-1}-\alpha p^{m-1}/8)(1+\alpha/4)p|C_1|^{r-1} \\ & \geq & \left(1+\alpha/10\right)p^{m}|C_1|^{r-1}.\end{eqnarray*}
The number of mappings from $H$ to $C_1 \cup D_1$ with vertex $i$ going to $v$ and the other $r-1$ vertices going to $C_1$ which are not one-to-one is at most $r^2|C_1|^{r-2}$. As $|C_1| \geq 20 r^2 \alpha^{-1}p^{-m}$, then we get at least  $ \left(1+\alpha/20\right)p^{m}|C_1|^{r-1}$ labeled copies of $H$ with vertex $i$ mapping to $v$ and the remaining vertices mapping to $C_1$. Summing over all choices of $v$ in $D_1$, we get that there are at least 
$ \left(1+\alpha/20\right)p^{m}|C_1|^{r-1}|D_1|$ labeled copies of $H$ with vertex $i$ mapping to $D_1$ and the remaining $r-1$ vertices mapping to $C_1$. Finally, summing over all $r$ choices of $i$, we get at least $r \left(1+\alpha/20\right)p^{m}|C_1|^{r-1}|D_1|$ labeled copies of $H$ with one vertex mapping to $D_1$ and the remaining $r-1$ vertices mapping to $C_1$. 

Finally, the number of possible mappings from $H$ to $C_1 \cup D_1$ with at least two vertices in $D_1$ is at most $\sum_{j \geq 2}{r \choose j}|D_1|^j|C_1|^{r-j} \leq r^2|D_1|^2|C_1|^{r-2}$. Putting the bounds together, we get that the number of labeled copies of $H$ in $C_1 \cup D_1$ is $p^m|C_1 \cup D_1|^r$ (this is the sum of the contributions of the main terms) plus at least 
$$r\frac{\alpha}{20}p^{m}|C_1|^{r-1}|D_1|-(mp^{m-1}\kappa+m\eta+r^2|C_1|^{-1})|C_1|^r-r^2|D_1|^2|C_1|^{r-2},  
$$
which, substituting in $|D_1|=\frac{\alpha p^{m}}{40r}|C_1|$, is equal to
$$\frac{\alpha^2}{1600}p^{2m}|C_1|^r-(mp^{m-1}\kappa+m\eta+r^2|C_1|^{-1})|C_1|^r.$$
Using $|C_1| \geq r^2 \eta^{-1}$ (recall that $\eta = 10^{-5}\alpha^2p^{2m}m^{-2}$) and substituting in $\kappa=2m\eta p^{1-m}$, we get that this is at least 
$$\frac{\alpha^2}{1600}p^{2m}|C_1|^r-4m^2\eta|C_1|^r \geq 10^{-5}\alpha^2p^{2m}|C_1|^r \geq \delta n^r,$$
provided $c'(p,r)$ is chosen sufficiently small. This completes the proof. 
\qed 

\section{Concluding remarks} \label{sec:conclude}

It is plausible that Theorem~\ref{complete-graphs} can be extended to all $H$, that is, that a graph $G$ which satisfies ${\cal P}^*_{H,p}(\delta)$ also satisfies ${\cal P}^*_{2,p}(\epsilon)$ with $\epsilon \leq c(p,r) \delta$, where $r$ is the number of vertices in $H$. However, it seems that new ideas will be needed to prove this in full generality. It would already be interesting to obtain a linear dependence in the special case $H = C_4$.

One might also ask about the quantitative aspects of other quasirandom equivalences. For example, we know that for any $\epsilon > 0$, there exists $\delta > 0$ such that if a graph $G$ has density $p$ and satisfies $\mathcal{P}_{C_4,p}(\delta)$, then it also satisfies $\mathcal{P}_{2, p}^*(\epsilon)$. The forcing conjecture, which was already mentioned in the introduction, states that a similar result should hold when $C_4$ is replaced by any bipartite graph $H$ which contains a
cycle. Somewhat tentatively, we are willing to venture that the following stronger quantitative version of this conjecture holds.

\begin{conjecture}\label{depend}
Let $H$ be a fixed bipartite graph of girth $g$ and $0 < p < 1$ a fixed constant. For each $\epsilon>0$, there is $\delta=\Omega(\epsilon^g)$ such that any graph $G$ with density $p$ which satisfies $\mathcal{P}_{H,p}(\delta)$ also satisfies $\mathcal{P}_{2,p}^*(\epsilon)$.
\end{conjecture}

To see that the girth dependence would be tight, consider a random graph with $n$ vertices whose vertex set is partitioned into two parts $V_1$ and $V_2$, each of order $n/2$, with density $p-\epsilon$ inside parts and density $p+\epsilon$ between parts. Equivalently, consider the generalized random graph $G$ on {\it two} vertices, where loops have weight $p-\epsilon$ and the edge between the two vertices has weight $p+\epsilon$. By picking one vertex from $G$, we see that it does not satisfy $\mathcal{P}_{2,p}^*(\epsilon/2)$.

Consider now a random mapping of the vertices of $H$ to the two vertices of $G$. For an edge $e$ of $H$, let $X_e=-1$ if both vertices of $e$ map to the same vertex of $G$ and $X_e=1$ if the vertices of $e$ map to different vertices of $G$. The homomorphism density of $H$ in $G$ is then
$$\mathbb{E}[\prod_{e}(p+X_e\epsilon)].$$
Suppose that $e_1,\ldots,e_k$ are edges of $H$ with $k < g$. Since these edges form a forest, it follows that
$$\mathbb{E}[X_{e_1}X_{e_2}\ldots X_{e_k}]=\mathbb{E}[X_{e_1}]\cdots\mathbb{E}[X_{e_k}].$$ 
In particular, since $\mathbb{E}[X_{e}]=0$,
this implies that the coefficient of $\epsilon^i$ is zero for $i=1,\ldots,g-1$ and, therefore, $G$ satisfies $\mathcal{P}_{H,p}(\delta)$ with $\delta = O(\epsilon^g)$.

It is not hard to verify that Conjecture~\ref{depend} holds when $H$ is an even cycle. Combining this observation with other known results allows us to prove the conjecture for some reasonably broad classes of graphs. For example, Theorem 1.1 in~\cite{CFS10} implies that if $H$ is a bipartite graph with $m$ edges which has two vertices in one part complete to the other part and minimum degree at least two in the first part, then the homomorphism density $t_H(G)$ satisfies $t_H(G) \geq t_{C_4}(G)^{m/4}$. Therefore, if $t_H(G) \leq p^m(1+ \epsilon^4)$, we have $t_{C_4}(G) \leq p^4 (1+ \epsilon^4)^{4/m} \leq p^4 (1 + \frac{8}{m} \epsilon^4)$ and the required result for $H$ follows from the $C_4$ case. As with the forcing conjecture, similar arguments can likely prove Conjecture~\ref{depend} for many of the graphs for which Sidorenko's conjecture is known to hold. On the other hand, our second-order Sidorenko conjecture may be easier to disprove than the conjecture itself.

\end{document}